\documentclass[11pt]{amsart} 
\usepackage{verbatim, latexsym, amssymb, amsmath,color}
\usepackage{epsfig}
\usepackage{hyperref}
\usepackage{float}
\usepackage{graphicx}
\usepackage{sidecap}

\def\Xint#1{\mathchoice
   {\XXint\displaystyle\textstyle{#1}}%
   {\XXint\textstyle\scriptstyle{#1}}%
   {\XXint\scriptstyle\scriptscriptstyle{#1}}%
   {\XXint\scriptscriptstyle\scriptscriptstyle{#1}}%
   \!\int}
\def\XXint#1#2#3{{\setbox0=\hbox{$#1{#2#3}{\int}$}
     \vcenter{\hbox{$#2#3$}}\kern-.5\wd0}}

\def\dashint{\Xint-}

\newtheorem{thm}{Theorem}[section]

\newtheorem{lemm}[thm]{Lemma}
\newtheorem{cor}[thm]{Corollary}

\theoremstyle{remark}
\newtheorem{rmk}[thm]{Remark}
\theoremstyle{definition}

\title{On the two-systole of real projective spaces}
\author{Lucas Ambrozio and Rafael Montezuma}
\address{L. Ambrozio: Institute for Advanced Study \\ Princeton NJ
08540 USA}
\email{lambrozio@ias.edu}

\address{R. Montezuma: Mathematics Department, Princeton University \\ Fine Hall, Washington Road \\ Princeton NJ 08544-1000 USA}
\email{rcabral@princeton.edu}

\thanks{MSC classification codes: 53A10, 53C30. Keywords: two-systole, minimal two-spheres, homogeneous three-spheres, integral-geometric formula.}

\begin{document}

\begin{abstract}
  We establish an integral-geometric formula for minimal two-spheres inside homogeneous three-spheres, and use it to provide a characterisation of each homogeneous metric on the three-dimensional real projective space as the unique metric with the largest possible two-systole among metrics with the same volume in its conformal class. 
\end{abstract}

\maketitle

\section{Introduction}

\indent Let $\mathbb{RP}^3$ be the three-dimensional real projective space, and $\mathcal{F}$ denote the non-empty set consisting of all embedded surfaces in $\mathbb{RP}^3$ that are diffeomorphic to the two-dimensional projective plane $\mathbb{RP}^2$. Given a Riemannian metric $g$ on $\mathbb{RP}^3$, we define
\begin{equation*}
	\mathcal{A}(\mathbb{RP}^{3},g) = \inf_{\Sigma \in \mathcal{F}}area(\Sigma,g).
\end{equation*}
\noindent In this paper, the geometric invariant above will be called the \textit{two-systole} of $(\mathbb{RP}^3,g)$. The term has been used to name slightly different invariants in the literature, depending on the choice of the set $\mathcal{F}$  (\textit{cf}. \cite{Gro}, Sections 1 and 4.A.7, and \cite{Ber}, Section 5). The first systematic study of such invariants was done by Berger in \cite{Ber}, where he sought generalisations of Pu's inequality \cite{Pu} for the (one)-systole of real projective planes, \textit{i.e.} the smallest length of a non-trivial loop in $(\mathbb{RP}^2,g)$. Berger computed that the two-systole of the standard round metric $g_1$ on $\mathbb{RP}^3$, with constant sectional curvature one, is equal to $2\pi$ (see \cite{Ber}, Th\'eor\`eme 7.1). This number is precisely the area of the totally geodesic projective planes in $(\mathbb{RP}^3,g_1)$. \\
\indent In \cite{BraBreEicNev}, Bray, Brendle, Eichmair and Neves studied how the two-systole behaves under the Ricci flow, proving along the way a sharp upper bound for $\mathcal{A}(\mathbb{RP}^3,g)$ in terms of the minimum value of the scalar curvature of $(\mathbb{RP}^3,g)$ (see \cite{BraBreEicNev}, Theorems 1.1 and 1.2). An important part of their analysis was to show that the infimum defining $\mathcal{A}(\mathbb{RP}^{3},g)$ is actually attained by an embedded area-minimising projective plane in $(\mathbb{RP}^{3},g)$ (see \cite{BraBreEicNev}, Proposition 2.3). \\
\indent In this paper, we investigate how large can be the \textit{normalised two-systole},
\begin{equation*}
	\frac{\mathcal{A}(\mathbb{RP}^3,g)}{vol(\mathbb{RP}^3,g)^{\frac{2}{3}}},
\end{equation*}
among metrics inside a conformal class defined by a homogeneous metric on $\mathbb{RP}^3$, \textit{i.e.} a Riemannian metric whose isometry group acts transitively. The result we obtain is the following:

\begin{thm} \label{thm-main}
	\textit{Let $\overline{g}$ be a homogeneous Riemannian metric on $\mathbb{RP}^3$.  If $g$ is a Riemannian metric on $\mathbb{RP}^3$ that is conformal to $\overline{g}$, then
	\begin{equation*}
		\frac{\mathcal{A}(\mathbb{RP}^3,g)}{vol(\mathbb{RP}^3,g)^{\frac{2}{3}}} \leq \frac{\mathcal{A}(\mathbb{RP}^3,\overline{g})}{vol(\mathbb{RP}^3,\overline{g})^{\frac{2}{3}}}.
	\end{equation*}
	Moreover, equality holds if and only if $g$ is a constant multiple of $\overline{g}$.}
\end{thm}

\indent Our proof of Theorem \ref{thm-main} is based on the classification of immersed minimal two-spheres in homogeneous three-spheres $(S^3,g)$, which has been obtained by Meeks, Mira, P\'erez and Ros \cite{MeeMirPerRos}. In a few words, up to ambient isometries, there exists a unique immersed minimal sphere, which is actually embedded and invariant under the antipodal map (see Theorem \ref{thm-minimal-spheres} for a more detailed statement). Denoting by $\mathcal{G}^{+}$ the set of oriented minimal two-spheres in a homogeneous $(S^3,g)$, we verify that $\mathcal{G}^{+}$ can be identified with $S^3$ itself, and that the following integral-geometric formula holds:
\begin{equation} \label{eq-formula-intgeo}
	\int_{\mathcal{G}^{+}} \left(\dashint_{\Sigma} f dA_{g} \right) d\mathcal{G}^{+}_{g} = \int_{S^3} f dV_{g} \quad \text{for all} \quad f\in C^{0}(S^3).
\end{equation}
Formula \eqref{eq-formula-intgeo} is well-known in the case of the round three-sphere, where minimal two-spheres are the totally geodesic equators (\textit{cf}. Santal\'o \cite{San}). From this point, a proof of Theorem \ref{thm-main} can be given following essentially the same argument, based on the Uniformisation Theorem, used by Pu and Loewner to establish their theorems about systoles of projective planes and two-tori, respectively (see for example \cite{Gro}, Section 1.B). \\
\indent We remark that the relevance of integral-geometric formulae similar to \eqref{eq-formula-intgeo} in this sort of maximisation problem for one-systoles was already recognised, notably by Gromov and Bavard \cite{Bav}. \\

\indent Restricting our attention to the conformal class of the round metrics, we can thus state the following
\begin{cor}
	If $g$ is a Riemannian metric on $\mathbb{RP}^3$ that is conformal to the round metric $g_1$, then
	\begin{equation*}
		\mathcal{A}(\mathbb{RP}^3,g) \leq \frac{2}{{\sqrt[\leftroot{-1}\uproot{2}\scriptstyle 3]\pi}}vol(\mathbb{RP}^3,g)^{\frac{2}{3}},
	\end{equation*}
	and equality holds if and only if $g$ is a constant multiple of $g_1$.
\end{cor}

\indent More generally, it seems to be an arduous task to calculate the actual value of the two-systole of all homogeneous metrics, which belong to a two-parameter family up to scaling (see Section \ref{Sec-2}), except perhaps in the case of the family of Berger metrics $g_{\rho}$, $\rho>0$, where explicit formulae can be deduced and a numeric computation is feasible. In Section \ref{Sec-5}, we show that the normalised two-systole of $(\mathbb{RP}^3,g_\rho)$ attains a local strict \textit{minimum} at the round metric ($\rho=1$), and diverges to infinity as the parameter goes to either $0$ or $+\infty$. In a way, one can speak very concretely about systolic freedom (\textit{cf.} \cite{CroKat}, Section 4): the normalised two-systole considered here is unbounded on the space of Riemannian metrics on $\mathbb{RP}^3$, even among homogeneous metrics with positive Ricci curvature. \\

\indent In a companion paper \cite{AmbMon}, we study the \textit{Simon-Smith width} \cite{SimSmi} of three-spheres $(S^3,g)$. When a metric on $S^3$ is the pull-back of a metric $g$ on $\mathbb{RP}^3$ by the canonical projection $\pi : S^3 \rightarrow \mathbb{RP}^3$, and satisfies extra geometric assumptions, the width of $(S^3,\pi^{*}g)$ provides a lower bound to twice the value of the two-systole of $(\mathbb{RP}^3,g)$; for instance, this assertion holds for metrics admitting no stable minimal two-spheres. In \cite{AmbMon}, we interpret the integral-geometric formula \eqref{eq-formula-intgeo} as an evidence that the homogeneous metrics on the three-sphere should be local maxima of the normalised widths in their conformal classes as well. In fact, it was this expectation that led us to investigate the topics discussed here.

\section{Minimal two-spheres in homogeneous three-spheres} \label{Sec-2}

\indent Let $S^3$ denote the unit sphere in $\mathbb{R}^4\simeq \mathbb{C}^2$, centred at the origin,
\begin{equation*}
	S^3 = \{(z,w) \in \mathbb{C}^2,\, |z|^2+|w|^2 =1\}.
\end{equation*}
\indent The three-sphere $S^3$ can be identified with the Lie group $SU(2)$ of the special unitary transformations of $\mathbb{C}^2$, which are represented by the two-by-two complex matrices of the form
\begin{equation*}
	\left[ \begin{matrix}
z & -\overline{w}  \\
w & \overline{z} \end{matrix} \right] ,\quad \text{where} \quad |z|^2 + |w|^2=1.
\end{equation*}
Under this identification, the group operation is given by
\begin{equation*}
	(z,w)\cdot(u,v) = (zu-\overline{w}v,wu+\overline{z}v).
\end{equation*} 
\indent The left (respect. right) multiplication by an element $(z,w)\in S^3$ will be denoted by $\mathcal{L}_{(z,w)}$ (respect. $\mathcal{R}_{(z,w)}$). Notice that $\mathcal{L}_{(1,0)} : S^3 \rightarrow S^3$ is the identity map, whereas $\mathcal{L}_{(-1,0)} : S^3 \rightarrow S^3$ is the antipodal map. \\
\indent The antipodal map commutes with all left translations. We can identify the quotient of $S^3$ by the antipodal map, \textit{i.e.} the three-dimensional real projective space $\mathbb{RP}^3$, with the Lie group $SO(3)$ of the special orthogonal transformations of $\mathbb{R}^3$. \\ 

\indent For every Riemannian metric $g$ on $S^3$ that is invariant under left translations there exists an orthonormal basis $\{E_1, E_2, E_3\}$ of left-invariant vector fields, and real numbers $c_1$, $c_2$ and $c_3$, such that
\begin{equation*}
	[E_2,E_3] = c_1E_1, \quad [E_3,E_1] = c_2E_2, \quad \text{and} \quad [E_1,E_2] = c_3E_3.
\end{equation*} 
See \cite{Mil}, Section 4, for more details. The canonical metric on $S^3$ corresponds to the parameters $c_1=c_2=c_3=2$. The Berger metrics $g_{\rho}$, where $\rho\neq 1$ is a positive real number, corresponds to the parameters $c_1=2\sqrt{\rho}$ and $c_2=c_3=2/\sqrt{\rho}$. Up to a choice of orientation of $S^3$, the constants $c_i$ can be taken to be all positive. \\
\indent The isometry group of a left-invariant metric in $S^3$ will contain transformations other than left translations; in particular, it can have dimension three, four (Berger metrics) or six (round metric). \\
\indent Any compact simply connected (locally) homogeneous Riemannian three-manifold is isometric to $S^3$ endowed with some left-invariant metric (see, for example, Theorem 2.4 in \cite{MeePer}). Since the antipodal map is a left translation, the pull-back by the canonical projection $\pi : S^3 \rightarrow \mathbb{RP}^3$ establishes a bijective correspondence between homogeneous metrics on $\mathbb{RP}^3$ and homogeneous metrics in $S^3$. We will therefore use the terms ``homogeneous" and ``left-invariant" in this paper interchangeably.  \\\\

\indent The geometry of immersed two-spheres with constant mean curvature in a homogeneous three-sphere has been extensively studied by Meeks, Mira, P\'erez and Ros \cite{MeeMirPerRos}. The next proposition summarises those properties of minimal two-spheres that we will need to know for the applications we have in mind:

\begin{thm}(cf. \cite{MeeMirPerRos}, Theorems 1.3 and 7.1) \label{thm-minimal-spheres}
	\\ \indent Let $g$ be a left-invariant metric on $S^3$. 
	\begin{itemize}
		\item[$i)$] There exists an embedded index one minimal sphere $\Sigma_0$ in $(S^3,g)$.
		\item[$ii)$] Every immersed minimal sphere in $(S^3,g)$ is a left translation of $\Sigma_0$. In particular, every immersed minimal sphere in $(S^3,g)$ is an embedded index one minimal sphere isometric to $\Sigma_0$.
		\item[$iii)$] The antipodal map leaves every minimal sphere in $(S^3,g)$ invariant.
		\item[$iv)$] If a left translation leaves a minimal sphere invariant, then it is either the identity map or the antipodal map.
	\end{itemize}		
\end{thm}

\indent The results of \cite{MeeMirPerRos} are actually much more general and detailed, and the interested reader is encouraged to study their paper. For the sake of convenience, we will briefly sketch some steps of the proof of the above statement here, taking a slightly different path than the one described in the aforementioned work. In particular, based on the recent progress on min-max theory, one can prove $i)$ directly; properties $iii)$ and $iv)$, which are important for us later, will be explained by different arguments.
\begin{proof}
	First, we observe that, due to homogeneity of the metric $g$, any immersed two-sphere $\Sigma$ in $(S^3,g)$ have nullity three (see Section 4 in \cite{MeeMirPerRos}).	In particular, zero is an eigenvalue of the Jacobi operator of $\Sigma$ that cannot be the first. Thus, no immersed minimal sphere in $(S^3,g)$ is stable. It follows that any left-invariant metric on $S^3$ satisfies the assumptions of the min-max Theorem 3.4 of Marques and Neves \cite{MarNev-Duke}. Hence, there exists an embedded index one minimal two-sphere $\Sigma_0$ in $(S^3,g)$, confirming $i)$. \\
\indent The next (and most important) step involves the study of the left-invariant Gauss map of an index one minimal two-sphere. A key point is to show that this map must be a diffeomorphism, and here the index one property is used in a crucial way. Then, a Hopf differential type argument is used to prove that all immersed minimal spheres are actually obtained by a left translation of the index one minimal sphere $\Sigma_0$. Because details are involved, we refer the reader to the proof of items $(1)$ and $(2)$ of Theorem 4.1 in \cite{MeeMirPerRos} (the paper \cite{DanMir} by Daniel and Mira contains an insightful discussion on the ideas at the origin of the argument). \\
\indent Recall that the antipodal map $\mathcal{L}_{(-1,0)}$ commutes with all left translations. Thus, in view of $ii)$, in order to prove $iii)$ it is enough to show that $(S^3,g)$ contains a minimal two-sphere that is invariant under the antipodal map. As the antipodal map is an isometry of $(S^3,g)$, we can pass to the quotient and look for minimal projective planes contained in $(\mathbb{RP}^3,g)$, for the inverse image of any such surface will be a minimal sphere in $(S^3,g)$ with the required property. The existence of such surface can be shown, for example, by using Meeks-Simon-Yau Theorem to find the element of $\mathcal{F}$ with the least possible area, as indicated on Remark 7.2 in \cite{MeeMirPerRos} (a detailed argument is presented in \cite{BraBreEicNev}, Proposition 2.3). \\
\indent Finally, we prove item $iv)$ as follows. Fix an orientation of $\Sigma_0$ by defining a normal unit vector field $N$. Let $G : \Sigma_0 \rightarrow S^2$ denote the left-invariant Gauss map of $\Sigma_0$: it assigns to each point $p\in \Sigma_0$ the unique unit vector $G(p)$ in $(T_{(1,0)}S^3,g)$ such that $D\mathcal{L}_{p}(G(p)) = N(p)\in T_{p}S^3$. As observed above, $G$ is a diffeomorphism. Moreover, it is immediate to check that $G(\mathcal{L}_{(-1,0)}(p))= - G(p)$ for every $p\in \Sigma_{0}$.  \\
\indent If $\mathcal{L}_{(a,b)}(\Sigma_0)=\Sigma_0$, then the map $\Phi = G\circ \mathcal{L}_{(a,b)}\circ G^{-1}$ is a diffeomorphism of $S^2 \subset (T_{(1,0)}S^3,g)$, which we can use to define the vector field
\begin{equation*}
	X : q \in S^{2} \mapsto \Phi(q) - g(\Phi(q),q)q \in T_{(1,0)}S^3.
\end{equation*}
For every $q\in S^2$, the vector $X(q)$ is tangent to $S^2$. Therefore there exists $q_0$ such that $X(q_0)=0$. By Cauchy-Schwartz, it is immediate to conclude that either $\Phi(q_0)=q_0$ or $\Phi(q_0)=-q_0$. In the first case, $\mathcal{L}_{(a,b)}$ has a fixed point, and thus $(a,b)=(1,0)$. In the second case, the composition $\mathcal{L}_{(-a,-b)}=\mathcal{L}_{(-1,0)}\circ \mathcal{L}_{(a,b)}$ has a fixed point, because $(G\circ\mathcal{L}_{(-a,-b)}\circ G^{-1})(q_0)=(G\circ \mathcal{L}_{(-1,0)} \circ G^{-1})(\Phi(q_0))= -\Phi(q_0) = q_0$. It follows that $(-a,-b)=(1,0)$, or equivalently $(a,b)=(-1,0)$, as claimed. 
\end{proof}

\begin{rmk} An immersed minimal two-sphere in the round three-sphere must be an embedded totally geodesic equator, as proven by Almgren \cite{Alm} and Calabi \cite{Cal} using the holomorphic differential technique pioneered by Hopf. Much later, Abresch and Rosenberg \cite{AbrRos} constructed a new holomorphic differential on surfaces in Berger spheres and used it to show that immersed minimal two-spheres are rotationally invariant and unique up to ambient isometries. In \cite{MeeMirPerRos}, Section 7, the (unique) minimal two-sphere in an arbitrary homogeneous three-sphere is constructed explicitly by geodesic reflection of certain Plateau discs along their boundaries, and their geometry is described in details.  
\end{rmk}

\section{The integral-geometric formula} \label{Sec-3}

\indent The result stated in the previous Section allows us to understand the space of all minimal two-spheres in a three-sphere endowed with a homogeneous metric $g$ completely. In fact, let $\mathcal{G}^{+}$ be set of all oriented immersed minimal spheres in $(S^3,g)$. By item $ii)$ and $iv)$ of Theorem \ref{thm-minimal-spheres},  $\mathcal{G}^{+}$ consists of embedded minimal spheres, and $S^3$ acts transitively and effectively on $\mathcal{G}^{+}$ by left translations . Thus, $\mathcal{G}^{+}$ can be identified with $S^{3}$ itself: choosing any $\Sigma_0$ in $\mathcal{G}^{+}$, the map 
\begin{equation*}
	(a,b) \in S^3 \mapsto \mathcal{L}_{(a,b)}(\Sigma_0) \in \mathcal{G}^{+}
\end{equation*}
is a bijection. We use this map to endow $\mathcal{G}^{+}$ with the Riemannian metric $g$ and all derived structures (metric, topology, volume element). Notice in particular that the natural topology of $\mathcal{G}^{+}$ (smooth graphical convergence) coincides with the topology induced by the above identification.\\
\indent In the next theorem, we prove the integral-geometric formula \eqref{eq-formula-intgeo}.

\begin{thm} \label{thm-integral-formula}
	Let $g$ be a homogeneous Riemannian metric on $S^3$, and $\mathcal{G}^{+}$ denote the set of all oriented minimal two-spheres in $(S^3,g)$. For every continuous function $f$ on $S^3$, the following formula holds:
	\begin{equation} \label{eq-formula}
		\int_{\mathcal{G}^{+}} \left(\dashint_{\Sigma} f dA_{g} \right) d\mathcal{G}^{+}_{g} = \int_{S^3} f dV_{g}.
	\end{equation}		
\end{thm}

\begin{proof}
	 Let $f$ be a continuous function on $S^3$, and fix $\Sigma_0$ in $\mathcal{G}^{+}$. Let $(a,b)$ be the unique point in $S^{3}$ such that $\mathcal{L}_{(a,b)}(\Sigma_0)=\Sigma$. Since $\mathcal{L}_{(a,b)}$ is an orientation-preserving isometry, we can compute the integral of $f$ over $\Sigma$ by
	 \begin{equation*}
	 	\int_{\Sigma} f dA_g = \int_{\Sigma_0} f(\mathcal{L}_{(a,b)}(p,q))dA_{g}(p,q).
	 \end{equation*}
In the above formula, $(p,q)$ denotes the integration variable, which is an arbitrary point of $\Sigma_0$. \\
\indent Clearly, $\Sigma$ and $\Sigma_0$ have the same area in $(S^3,g)$. Given the identification between $\mathcal{G}^{+}$ and $S^3$, we can now use Fubini's Theorem to compute
\begin{align*}
	\int_{\mathcal{G}^{+}} \left(\dashint_{\Sigma} f dA_{g} \right) d\mathcal{G}^{+}_{g} & = \int_{S^3} \left(\dashint_{\Sigma_0} f(\mathcal{L}_{(a,b)}(p,q))dA_{g}(p,q) \right) dV_{g}(a,b) \\
	                 & = \dashint_{\Sigma_0} \left( \int_{S^3} f(\mathcal{L}_{(a,b)}(p,q))dV_{g}(a,b) \right) dA_{g}(p,q) \\
	                 & = \dashint_{\Sigma_0} \left( \int_{S^3} f(\mathcal{R}_{(p,q)}(a,b))dV_{g}(a,b) \right)dA_{g}(p,q).
\end{align*}
\indent As any compact Lie group, $S^3$ is unimodular: the volume form $dV_g$ of the left-invariant metric $g$ must be also invariant by right translations. Thus, for every $(p,q)$ in $\Sigma_0$,
\begin{equation*}
	\int_{S^3} f(\mathcal{R}_{(p,q)}(a,b))dV_{g}(a,b) = \int_{S^3} f dV_g.
\end{equation*}
\noindent Therefore
\begin{equation*}
	\int_{\mathcal{G}^{+}} \left(\dashint_{\Sigma} f dA_{g} \right) d\mathcal{G}^{+}_{g} = \dashint_{\Sigma_0} \left( \int_{S^3} f dV_g\right) dA_g(p,q) = \int_{S^3} f dV_g.
\end{equation*}
\end{proof}

\section{The two-systole of homogeneous metrics} \label{Sec-4}

\indent The next Lemma is a direct consequence of formula \eqref{eq-formula}.

\begin{lemm} \label{lemma}
	Let $\overline{g}$ be a homogeneous Riemannian metric on $S^{3}$, $\mathcal{G}^{+}$ denote the set of all oriented minimal spheres in $(S^3,\overline{g})$, and $w(\overline{g})$ be the common value of the area of each element in $\mathcal{G}^+$. If $g$ is a Riemannian metric on $S^{3}$ that is conformal to $\overline{g}$, then 
	\begin{equation*}
		\min_{\Sigma \in \mathcal{G}^{+}}area(\Sigma,g) \leq \frac{w(\overline{g})}{vol(S^{3},\overline{g})^{\frac{2}{3}}} vol(S^{3},g)^{\frac{2}{3}}.		
	\end{equation*}
	Moreover, equality holds if and only if $g$ is a constant multiple of $\overline{g}$.
\end{lemm}

\begin{proof}
	\indent Write $g=\phi \overline{g}$ for some positive smooth function $\phi$ on $S^{3}$. For every $\Sigma$ in $\mathcal{G}^{+}$, we have
	\begin{equation*}
		area(\Sigma,g) = \int_{\Sigma} \phi dA_{\overline{g}} \quad \Rightarrow \quad area(\Sigma,g) = w(\overline{g})\dashint_{\Sigma} \phi dA_{\overline{g}}.
	\end{equation*}		
	Therefore, the integral-geometric formula \eqref{eq-formula} gives
	\begin{multline*}
		\dashint_{\mathcal{G}^{+}}area(\Sigma,g)dV_{\overline{g}} = w(\overline{g})\dashint_{S^{3}}\phi dV_{\overline{g}} \\
\leq w(\overline{g})\left(\dashint_{S^{3}}\phi^{\frac{3}{2}} dV_{\overline{g}}\right)^{\frac{2}{3}} = \frac{w(\overline{g})}{vol(S^{3},\overline{g})^{\frac{2}{3}}}vol(S^3,g)^{\frac{2}{3}}.
	\end{multline*}
	where we used H\"older's inequality and the fact that $dV_g=\phi^{\frac{3}{2}}dV_{\overline{g}}$. Equality holds if and only if $\phi$ is a positive constant. \\
	\indent Since $\mathcal{G}^{+}$ is compact and the map $\Sigma\in \mathcal{G}^{+} \mapsto area(\Sigma,g) \in \mathbb{R}$ is continuous, the theorem follows.	
\end{proof}

\begin{rmk} \label{rmk-value}
	From the proof of item $iii)$ of Theorem \ref{thm-minimal-spheres}, it should be clear that the value of $w(\overline{g})$ in Lemma \ref{lemma} is equal to twice the value of $\mathcal{A}(\mathbb{RP}^3,\widetilde{g})$, where $\widetilde{g}$ is the unique homogeneous metric on $\mathbb{RP}^3$ such that the canonical projection $\pi : (S^3,\overline{g}) \rightarrow (\mathbb{RP}^3,\widetilde{g})$ is a local isometry.
\end{rmk}

\indent We are now ready to prove Theorem \ref{thm-main}:

\begin{proof}
	By Theorem \ref{thm-minimal-spheres}, each minimal sphere in the homogeneous $(S^3,\pi^{*}\overline{g})$ is embedded and invariant under the antipodal map. Thus, every element of $\mathcal{G}^+$ projects down to $\mathbb{RP}^3$ as an element of $\mathcal{F}$. The result is now a direct consequence of the definition of the two-systole, Lemma \ref{lemma} and Remark \ref{rmk-value}.
\end{proof}		

\section{The two-systole of Berger metrics} \label{Sec-5}

\indent In this section, we compute the value of the normalised two-systole of Berger spheres. We follow the nice exposition of Torralbo in \cite{Tor} and \cite{Tor2}. Given $\rho>0$, the Berger metric on $S^3=\{(z,w)\in \mathbb{C}^2;\, |z|^2+|w|^2=1\}$ is defined by
\begin{equation*}
	g_{\rho}(X,Y) = \langle X, Y \rangle + (\rho^2-1)\langle X,\xi\rangle	\langle Y,\xi\rangle \quad \text{for all} \quad X, Y \in \mathcal{X}(S^3),
\end{equation*}
where $\langle-,-\rangle$ denotes the Eulidean metric on $\mathbb{R}^4\simeq \mathbb{C}^2$ and $\xi: (z,w) \in S^3 \mapsto (iz,iw) \in \mathbb{C}^2$ is the vector field generating the Hopf action of $S^1$ on $S^3$. Notice that $g_{\rho}(\xi,\xi)$ is constant and equal to $\rho^2$, and that $g_{\rho}$ coincides with the standard metric in the orthogonal complement of $\xi$. The metric $g_{1}$ is the standard metric on the unit three-sphere $S^3$. \\
\indent The volume of $(S^3,g_{\rho})$ is equal to 
\begin{equation*}
	vol(S^3,g_{\rho}) = \rho vol(S^3,g_1) = 2\pi^2\rho.
\end{equation*}
\indent For all values of $\rho$, the vector field $\xi$ is an eigenvector of the Ricci tensor of $g_{\rho}$ associated to the eigenvalue $2\rho^2$. When $\rho\neq 1$, the other eigenvalue has multiplicity two and is equal to $4-2\rho^2$. In particular, $(S^3,g_{\rho})$ has positive Ricci curvature when $0 < \rho < \sqrt{2}$.  \\
\indent As observed in \cite{Tor2}, Section 3, the horizontal two-sphere
\begin{equation*}
	\Sigma_0 = \{(z,w)\in S^3;\, w=\overline{w}\} 
\end{equation*}
is precisely the unique minimal two-sphere in $(S^3,g_{\rho})$ up to ambient isometries, for all values of $\rho>0$. An explicit formula for its area is given in \cite{Tor}, Proposition 2. We perform the computation differently: using standard polar coordinates $(s,\theta)$ based at the north pole $(0,0,1,0)$ to parametrised $\Sigma_0$, it is straightforward to calculate
	\begin{equation*}
		area(\Sigma_0,g_{\rho}) = 2\pi \int_{0}^{\pi} \sqrt{\sin^{2}(s)+(\rho^2-1)\sin^{4}(s)} ds. 
	\end{equation*}
\indent The two-systole of $(\mathbb{RP}^3,g_{\rho})$ is equal to half the area of $\Sigma_0$ in $(S^3,g_{\rho})$. Thus, up to the constant factor $1/{\sqrt[\leftroot{-1}\uproot{2}\scriptstyle 3]\pi}$, the normalised two-systole of $(\mathbb{RP}^3,g_{\rho})$ is computed by the function
\begin{equation*}
	F : \rho \in (0,+\infty) \mapsto \frac{1}{\rho^{\frac{2}{3}}}\int_{0}^{\pi}\sin(s) \sqrt{ (1-\sin^{2}(s)) + \rho^2\sin^{2}(s)} ds \in \mathbb{R}.
\end{equation*}	
\noindent It is possible to check that
\begin{equation*}
	F'(1)=0, \quad F''(1)>0 \quad \text{and} \quad  \lim_{\rho\rightarrow 0} F(\rho)= \lim_{\rho\rightarrow +\infty} F(\rho) = +\infty
\end{equation*}
rather easily. In  words: among Berger metrics, the normalised two-systole attains a strict local minimum at the round metric $g_{1}$ and diverges to infinity either as the size of the Hopf orbits increase beyond all bounds, or as they collapse to zero. \\

\noindent \textbf{Acknowledgements:} The investigations that led to the writing of this paper initiated while L.A. was a Research Fellow at the University of Warwick, supported by the EPSRC Programme Grant `Singularities of Geometric Partial Differential Equations', reference number EP/K00865X/1. L.A. would also like to thank Fernando Marques for the invitation to visit the University of Princeton in April 2018, during which the first conversations about this project between the authors took place.


\begin{thebibliography}{99}

\bibitem{AbrRos} U. Abresch and H. Rosenberg, \textit{Generalized Hopf differentials}, Matem. Contemp. \textbf{28} (2005), 1-28.

\bibitem{Alm} F. Almgren Jr, \textit{Some interior regularity theorems for minimal surfaces and an extension of Bernstein's theorem}, Ann. of Math., \textbf{85} (1966), 277-292.

\bibitem{AmbMon} L. Ambrozio and R. Montezuma, \textit{On the width of unit volume three-spheres}, arXiv: 1809.03638.

\bibitem{Bav} C. Bavard, \textit{In\'egalit\'es isosystoliques conformes}, Comment. Math. Helv., \textbf{67} (1992), no. 1, pp. 146-166.

\bibitem{Ber} M. Berger, \textit{Du c\^ot\'e de chez Pu}, Annal. Scient. de l'\'E.N.S., 4-\`eme s\'erie, tome  5, no 1 (1972), p. 1-44.

\bibitem{BraBreEicNev} H. Bray, S. Brendle, M. Eichmair and A. Neves, \textit{Area-minimizing projective planes in 3-manifolds}, Comm. Pure Appl. Math., \textbf{63} (2010), 1237-1247.

\bibitem{Cal} E. Calabi, \textit{Minimal immersions of surfaces in Euclidean spheres}, J. Diff. Geom. \textbf{1} (1967), 111-125

\bibitem{CroKat} C. Croke and M. Katz, \textit{Universal volume bounds in Riemannian manifolds}, in Surveys in Differential Geometry, \textbf{8} (2003), pp 109-137.

\bibitem{DanMir} B. Daniel and P. Mira, \textit{Existence and uniqueness of constant mean curvature spheres in $Sol_3$}, J. reine angew. Math. \textbf{685} (2013), 1-32.

\bibitem{Gro} M. Gromov, \textit{Systoles and intersystolic inequalities}. In \textit{Actes de la Table Ronde de G\'eom\'etrie Diff\'erentielle} (Luminy, 1992), pages 291-362. S\'emin. Congr., vol. 1, Soc. Math. France, Paris (1996). 

\bibitem{MarNev-Duke} F. C. Marques and A. Neves,
\textit{Rigidity of min-max minimal spheres in three-manifolds.} Duke Math. J. \textbf{161} (2012), no. 14, 2725-2752. 

\bibitem{MeeMirPerRos} W. Meeks III, P. Mira, J. P\'erez and A. Ros, \textit{Constant mean curvature spheres in homogeneous three-spheres}, arXiv: 1308.2612.

\bibitem{MeePer} W. Meeks III and J. P\'erez, \textit{Constant mean curvature surfaces in metric Lie groups}. In \textit{Geometric Analysis: Partial Differential Equations
and Surfaces}, volume 570, pages 25-110. Contemporary Mathematics AMS, edited by J. Galvez and J. P\'erez (2012).

\bibitem{Mil} J. Milnor, \textit{Curvatures of left invariant metrics on Lie groups}, Adv. Maths, \textbf{21} (1976), 293-329.

\bibitem{Pu} P. M. Pu, \textit{Some inequalities in certain nonorientable Riemannian manifolds}, Pacific J. Math. \textbf{2} (1952), no. 1, 55-71.

\bibitem{San} Santal\'o, \textit{Integral Geometry and Geometric Probability}, Cambridge University Press, 2nd Edition, 2004.

\bibitem{SimSmi} F. Smith,
\textit{On the existence of embedded minimal 2-spheres in the 3-sphere, endowed with an arbitrary Riemannian metric.} Ph.D. thesis, Supervisor: L. Simon (1982).

\bibitem{Tor} F. Torralbo, \textit{Rotationally invariant  constant  mean  curvature surfaces  in  homogeneous 3-manifolds}, Diff. Geom. and its Applications, \textbf{28} (2010), no. 5, pages 593-607.

\bibitem{Tor2} F. Torralbo, \textit{Compact minimal surfaces in the Berger spheres}, Ann. Glob. Anal. Geom., \textbf{41} (2012), no. 4, pp 391-405.

\end{thebibliography}
\end{document}